\documentclass{amsart} \usepackage{amsmath,amssymb,amsfonts}
\usepackage[pdftex]{graphicx}

\usepackage[colorlinks=false]{hyperref}
\hypersetup{colorlinks,citecolor=blue,linktocpage}

\usepackage{sty-va}

\date{November 23, 2013}

\title{Divisors on Burniat surfaces}
 \author{Valery
  Alexeev} \address{Department of Mathematics, University of Georgia,
  Athens, GA 30605, USA} \email{valery@math.uga.edu} 

\begin{document}

\begin{abstract}
  In this short note, we extend the results of [Alexeev-Orlov, 2012]
  about Picard groups of Burniat surfaces with $K^2=6$ to the cases of
  $2\le K^2\le 5$. We also compute the semigroup of effective divisors
  on Burniat surfaces with $K^2=6$. Finally, we construct an
  exceptional collection on a nonnormal semistable degeneration of a
  1-parameter family of Burniat surfaces with $K^2=6$.
\end{abstract}

\maketitle

\tableofcontents

\section*{Introduction}
\label{sec:intro}

This note strengthens and extends several geometric results of the
paper \cite{alexeev2012derived-categories}, joint with Dmitri Orlov,
in which we constructed exceptional sequences of maximal possible length
on Burniat surfaces with $K^2=6$. The construction was based on certain
results about the Picard group and effective divisors on Burniat
surfaces.

Here, we extend the results about Picard group to Burniat surfaces with
$2\le K^2\le 5$. We also establish a complete description of the
semigroup of effective $\bZ$-divisors on Burniat surfaces with
$K_X^2=6$. (For the construction of exceptional sequences in
\cite{alexeev2012derived-categories} only a small portion of this
description was needed.)

Finally, we construct an exceptional collection on a nonnormal
semistable degeneration of a 1-parameter family of Burniat surfaces
with $K^2=6$.

\section{Definition of Burniat surfaces}

In this paper, Burniat surfaces will be certain smooth surfaces of
general type with $q=p_g=0$ and $2\le K^2\le 6$ with big and nef
canonical class $K$ which were defined by Peters in
\cite{peters1977on-certain-examples} following Burniat.  They are
Galois $\bZ_2^2$-covers of (weak) del Pezzo surfaces with $2\le K^2\le
6$ ramified in certain special configurations of curves.

Recall from \cite{pardini1991abelian-covers} that a $\bZ_2^2$-cover 
$\pi\colon X\to Y$ with smooth and projective $X$ and $Y$ is determined by
three branch divisors $\bar A,\bar B,\bar C$ and three invertible
sheaves $L_1, L_2, L_3$ on the base $Y$ satisfying fundamental
relations $L_2\otimes L_3\simeq L_1(\bar A)$, $L_3\otimes L_1\simeq
L_2(\bar B)$, $L_1\otimes L_2 \simeq L_3(\bar C)$. These relations
imply that $L_1^2 \simeq \cO_Y(\bar B+\bar C)$, $L_2^2 \simeq
\cO_Y(\bar C+\bar A)$, $L_3^2 \simeq \cO_Y(\bar A+\bar B)$.

One has $X=\Spec_Y \cA$, where the $\cO_Y$-algebra $\cA$ is
$\cO_Y\oplus\oplus_{i=1}^3 L_i\inv$. The multiplication is determined
by three sections in
\begin{math}
  \Hom( L_i\inv \otimes L_j\inv, L_k\inv) =
  H^0( L_i\otimes L_j\otimes L_i\inv),
\end{math}
where $\{i,j,k\}$ is a permutation of $\{1,2,3\}$, i.e. by sections of
the sheaves $\cO_Y(\bar A)$, $\cO_Y(\bar B)$, $\cO_Y(\bar C)$
vanishing on $\bar A$, $\bar B$, $\bar C$.

Burniat surfaces with $K^2=6$ are defined by taking $Y$ to be the
del Pezzo surface of degree 6, i.e. the blowup of $\bP^2$ in three
noncollinear points, and the divisors
$\bar A=\sum_{i=0}^3 \bar A_i$, $\bar B=\sum_{i=0}^3 \bar B_i$, $\bar
C=\sum_{i=0}^3 \bar C_i$ to be the ones shown in red, blue, and black
in the central picture of Figure~\ref{fig:Burniat} below. 

The divisors $\bar A_i,\bar B_i,\bar C_i$ for $i=0,3$ are the
$(-1)$-curves, and those for $i=1,2$ are $0$-curves, fibers of
rulings $\Bl_3\bP^2\to \bP^1$.
The del Pezzo surface also has two contractions to $\bP^2$
related by a quadratic transformation, and the images of the divisors
form a special line configuration on either $\bP^2$.
We denote the fibers of the three rulings $f_1,f_2,f_3$ and the
preimages of the hyperplanes from $\bP^2$'s by $h_1,h_2$. 

\begin{figure}[h]
  \centering
  \includegraphics[height=1.5in]{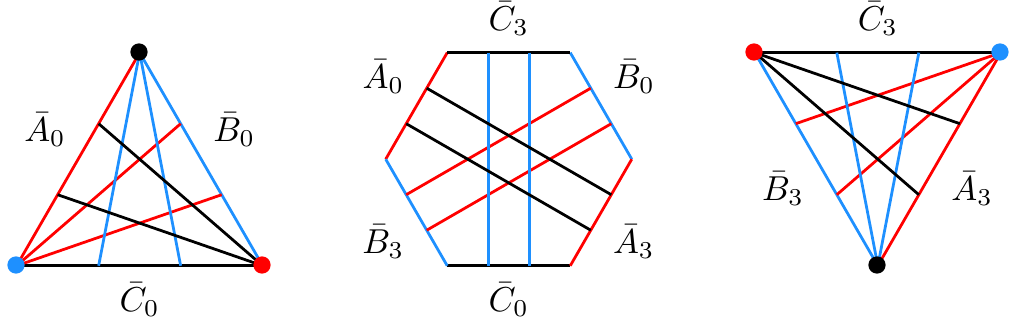}
  \caption{Burniat configuration on $\Bl_3\bP^2$}
  \label{fig:Burniat}
\end{figure}

Burniat surfaces with $K^2=6-k$, $1\le k\le 4$ are obtained by
considering a special configuration in Figure~\ref{fig:Burniat} for
which some $k$ triples of curves, one from each group $\{\bar A_1, \bar
A_2\}$, $\{\bar B_1, \bar B_2\}$, $\{\bar C_1, \bar C_2\}$, meet at
common points $P_s$. The corresponding Burniat surface is the
$\bZ_2^2$-cover of the blowup of $\Bl_3\bP^2$ at these points. 

Up to symmetry, there are the following cases, see
\cite{bauer2011burniat-surfaces1}:
\begin{enumerate}
\item $K^2=5$: $P_1 = \bar A_1\bar B_1\bar C_1$ (our shortcut notation
  for   $\bar A_1\cap \bar B_1\cap \bar C_1$). 
\item $K^2=4$, nodal case: $P_1 = \bar A_1\bar B_1\bar C_1$, 
  $P_2 = \bar A_1\bar B_2\bar C_2$. 
\item $K^2=4$, non-nodal case: $P_1 = \bar A_1\bar B_1\bar C_1$, 
  $P_2 = \bar A_2\bar B_2\bar C_2$. 
\item $K^2=3$: $P_1 = \bar A_1\bar B_1\bar C_2$, 
  $P_2 = \bar A_1\bar B_2\bar C_1$, $P_3=\bar A_2\bar B_1\bar C_1$. 
\item $K^2=2$: $P_1 = \bar A_1\bar B_1\bar C_1$, 
  $P_2 = \bar A_1\bar B_2\bar C_2$, $P_3=\bar A_2\bar B_1\bar C_2$, 
  $P_4 = \bar A_2\bar B_2\bar C_1$. 
\end{enumerate}

\begin{notation}
  We generally denote the divisors upstairs by $D$ and the divisors
  downstairs by $\bar D$ for the reasons which will become clear from
  Lemmas~\ref{lem:lattices}, \ref{lem:lattices2-5}. We denote
  $Y=\Bl_3\bP^2$ and $\eps\colon Y'\to Y$ is the blowup map at the
  points $P_s$. The exceptional divisors are denoted by $\bar E_s$.

  The curves $\bar A_i,\bar B_i,\bar C_i$ are the curves on $Y$, the
  curves $\bar A'_i,\bar B'_i,\bar C'_i$ are their strict preimages
  under $\eps$. (So that $\eps^*(\bar A_1)= \bar A'_1+E_1$ in the case
  (1), etc.)  The divisors $A'_i,B'_i,C'_i,E_s$ are the curves (with
  reduced structure) which are the preimages of the latter curves and
  $\bar E_s$ under $\pi'\colon X'\to Y'$. The surface $X'$ is the
  Burniat surface with $K^2=6-k$.
\end{notation}

The building data for the $\bZ_2^2$-cover $\pi'\colon X'\to Y'$
consists of three divisors $A'=\sum \bar A'_i$, $B'=\sum \bar B'_i$,
$C'=\sum \bar C'_i$. It does \emph{not} include the exceptional
divisors $\bar E_s$, they are not in the ramification locus.

One has $\pi'^* (\bar A'_i) = 2A'_i$, $\pi'^* (\bar B'_i) = 2B'_i$, $\pi'^*
(\bar C'_i) = 2C'_i$,  and $\pi'^* (\bar E_s) = E_s$. 

\smallskip

For the canonical class, one has $2K_X = \pi^*(-K_Y)$.  Indeed, from
Hurwitz formula $2K_{X'}=\pi^*(2K_{Y'}+R')$, where
$R'=A'+B'+C'$. Therefore, the above identity is equivalent to
$R'=-3K_{Y'}$. This holds on $Y=\Bl_3\bP^2$, and
  \begin{displaymath}
    R'= \eps^*R -3\sum \bar E_s = \eps^*(-3K_{Y})-3\sum \bar E_s = -3 K_{Y'}.
  \end{displaymath}

  For the surfaces with $K^2=6,5$ and $4$ (non-nodal case), $-K_Y$ and
  $K_X$ are ample. For the remaining cases, including $K^2=2,3$, the
  divisors $-K_Y$ and $K_X$ are big, nef, but not ample.  Each of the
  curves $\bar L_j$ (among $\bar A_i,\bar B_i,\bar C_i$) through two
  of the points $P_s$ is a $(-2)$-curve (a $\bP^1$ with square $-2$)
  on the surface $Y$. (For example, for the nodal case with $K^2=4$
  $\bar L_1=\bar A_1$ is such a line). Its preimage, a curve $L_j$ on $X$,
  is also a $(-2)$-curve. One has $-K_Y\bar L_j = K_XL_j=0$, and the
  curve $L_j$ is contracted to a node on the canonical model of $X$.

Note that both of the cases with $K^2=2$ and $3$ are nodal.

\section{Picard group of Burniat surfaces with  $K^2=6$}

In this section, we recall two results of
\cite{alexeev2012derived-categories}.

\begin{lemma}[\cite{alexeev2012derived-categories}, Lemma 1]
\label{lem:lattices}
The homomorphism $\bar D\mapsto \frac12\pi^*(\bar D)$ defines an
  isomorphism of integral lattices $\frac12\pi^*\colon\Pic Y \to \Pic
  X/\Tors$. Under this isomorphism, one has $\frac12\pi^*(-K_Y) = K_X$.
\end{lemma}

This lemma allows one to identify $\bZ$-divisors $\bar D$ on the
del Pezzo surface $Y$ with classes of $\bZ$-divisors $D$ on $X$ up
to torsion, equivalently up to numerical equivalence. This
identification preserves the intersection form.

\medskip

The curves $A_0,B_0,C_0$ are elliptic curves (and so are the curves
$A_3\simeq A_0$, etc.). Moreover, each of them comes with a canonical
choice of an origin, denoted $P_{00}$, which is the point of
intersection with the other curves which has a distinct color,
different from the other three points. (For example, for $A_0$ one has
$P_{00}= A_0\cap B_3$.)

On the elliptic curve $A_0$ one also defines $P_{10}=A_0\cap C_3$,
$P_{01}=A_0\cap C_1$, $P_{11}=A_0\cap C_2$. This gives the 4 points in
the 2-torsion group $A_0[2]$. We do the same for $B_0$, $C_0$
cyclically. 

\begin{theorem}\label{thm:PicX}
  [\cite{alexeev2012derived-categories}, Theorem 1]
  One has the following:
  \begin{enumerate}
  \item The homomorphism
    \begin{eqnarray*}
      \phi\colon
      \Pic X &\to& \bZ \times \Pic A_0 \times \Pic B_0 \times \Pic C_0 \\
      L&\mapsto& (d(L)=L\cdot K_X,\  L|_{A_0},\  L|_{B_0},\  L|_{C_0}
      )
    \end{eqnarray*}
    is injective, and the image is the subgroup of index 3 of
    \begin{displaymath}
      \bZ\times
      (\bZ.P_{00} + A_0[2])
      \times (\bZ.P_{00} + B_0[2])
      \times (\bZ.P_{00} + C_0[2])
      \simeq \bZ^4 \times \bZ_2^6.
    \end{displaymath}
    consisting of the elements with $d+a_0^0+b_0^0+c_0^0$ divisible by
    3. Here, we denote an element of the group $\bZ.P_{00} + A_0[2]$ by 
    $(a_0^0\ a_0^1a_0^2)$, etc. 

  \item $\phi$ induces an isomorphism $\Tors(\Pic X)\to A_0[2]\times
    B_0[2]\times C_0[2]$.
  \item The curves $A_i,B_i,C_i$, $0\le i\le 3$, generate $\Pic X$.
  \end{enumerate}
\end{theorem}

This theorem provides one with explicit coordinates for the Picard
group of a Burniat surface $X$, convenient for making computations.

\section{Picard group of Burniat surfaces with  $2\le K^2\le 5$}

In this section, we extend the results of the previous section to the
cases $2\le K^2\le 5$. 
First, we show that Lemma~\ref{lem:lattices}
holds verbatim if $3\le K^2 \le 5$. 

\begin{lemma}\label{lem:lattices2-5}
  Assume $3\le K^2\le 5$. Then
  the homomorphism $\bar D\mapsto \frac12\pi'^*(\bar D)$ defines an
  isomorphism of integral lattices $\frac12\pi'^*\colon\Pic Y' \to \Pic
  X'/\Tors$, and the inverse map is $\frac12\pi'_*$. 
  Under this isomorphism, one has $\frac12\pi'^*(-K_{Y'}) = K_{X'}$.
\end{lemma}
\begin{proof}
  The proof is similar to that of Lemma~\ref{lem:lattices}. The map
  $\frac12\pi^*$ establishes an 
  isomorphism of $\bQ$-vector spaces $(\Pic Y')\otimes\bQ$ and $(\Pic
  X')\otimes\bQ$ together with the intersection product because:
  \begin{enumerate}
  \item Since $h^i(\cO_{X'})=h^i(\cO_{Y'})=0$ for $i=1,2$ and $K_{X'}^2=K_{Y'}^2$,
    by Noether's formula the two vector spaces have the same dimension.
  \item $\frac12\pi'^*\bar D_1\cdot\frac12\pi'^*\bar D_2 =
    \frac14\pi'^*(\bar D_1\cdot \bar D_2)=\bar D_1\bar D_2$.
  \end{enumerate}
  A crucial observation is that $\frac12\pi'^*$ sends $\Pic Y'$ to
  integral classes. To see this, it is sufficient to observe that $\Pic
  Y'$ is generated by divisors $\bar D$ which are in the ramification
  locus and thus for which $D=\frac12\pi'^*(\bar D)$ is integral. 

  Consider for example the case of $K^2=5$.  One has $\Pic
  Y'=\eps^*(\Pic Y)\oplus \bZ E$.  The group $\eps^*(\Pic Y)$ is
  generated by $\bar A'_0,\bar B'_0,\bar C'_0,\bar A'_3,\bar B'_3,\bar
  C'_3$. Since $\eps^*(\bar A_1) = \bar A_1' + \bar E_1$, the divisor
  class $\bar E_1$ lies in group spanned by $\bar A_1'$ and
  $\eps^*(\Pic Y)$. So we are done. 

  In the nodal case $K^2=4$, $\bar E_1$ is spanned by $\bar B_1'$ and
  $\eps^*(\Pic Y)$, $\bar E_2$ by $\bar B_2'$ and
  $\eps^*(\Pic Y)$; exactly the same for the non-nodal case. 
  In the case $K^2=3$, 
  $\bar E_1$ is spanned by $\bar C_2'$ and $\eps^*(\Pic Y)$, 
  $\bar E_2$ by $\bar B_2'$ and $\eps^*(\Pic Y)$, 
  $\bar E_3$ by $\bar A_2'$ and $\eps^*(\Pic Y)$.

  Therefore, $\frac12\pi'^*(\Pic Y')$ is a sublattice of finite index in
  $\Pic X'/\Tors$. Since the former lattice is unimodular, they must be
  equal. 

  One has $\frac12\pi'_* \circ \frac12\pi'^*(\bar D)=\bar D$, so the inverse
  map is $\frac12\pi'_*$. 
\end{proof}

\begin{remark}
  I thank Stephen Coughlan for pointing out that the above proof that
  $\Pic Y'$ is generated by the divisors in the ramification locus
  does not work in the $K^2=2$ case. In this case, each of the 
  lines $\bar A_i, \bar B_i, \bar C_i$, $i=1,2$ contains exactly two
  of the points $P_1,P_2,P_3$. What we can see easily is the
  following: there exists a free abelian group $H\simeq \bZ^8$ which
  can be identified with a subgroup of index 2 in $\Pic Y'$ and a
  subgroup of index 2 in $\Pic X'/\Tors$. 
\end{remark}

Consider a $\bZ$-divisor (not a divisor class) on $Y'$
$$\bar D=a_0\bar A'_0+\dotsc+c_3\bar C'_3+\sum_s e_s\bar E_s$$ such that the
coefficients $e_s$ of $\bar E_s$ are even. Then we can define a canonical
lift 
$$D=a_0A_0+\dotsc+c_3 C_3+ \sum_s \frac12 e_s E_s,$$ 
which is a divisor on $X'$, and numerically one has $D=\frac12\pi'^*(\bar D)$. 
Note that $\bar D$ is linearly equivalent to 0 iff $D$ is a torsion.
 
\medskip

By Theorem~\ref{thm:PicX}, for a Burniat surface with $K^2=6$, we have
an identification 
\begin{displaymath}
  V:= \Tors\Pic X = A_0[2] \times B_0[2] \times C_0[2] 
  = \bZ_2^2 \times\bZ_2^2 \times\bZ_2^2. 
\end{displaymath}

It is known (see \cite{bauer2011burniat-surfaces1}) that for Burniat
surfaces with $2\le K^2\le 6$ one has $\Tors\Pic X\simeq \bZ_2^{K^2}$
with the exception of the case $K^2=2$ where $\Tors\Pic X\simeq
\bZ_2^{3}$.  We would like to establish a convenient presentation for
the Picard group and its torsion for these cases which would be
similar to the above. 

\begin{definition}
  We define the following vectors, forming a basis in the
  $\bZ_2$-vector space $V$:
  $\vec A_1 = 00\ 10\ 00$, $\vec A_2 = 00\ 11\ 00$, $\vec B_1 = 00\
  00\ 10$, $\vec B_2 = 00\ 00\ 11$, $\vec C_1 = 10\ 00\ 00$, $\vec C_2
  = 11\ 00\ 00$.

  Further, for each point $P_s=A_iB_jC_k$ we define a vector $\vec P_s=
  \vec A_i + \vec B_j + \vec C_k$. 
\end{definition}

\begin{definition}
  We also define the standard bilinear form $V\times V\to\bZ_2$: \linebreak
  $(x_1,\dotsc,x_6) \cdot (y_1,\dotsc,y_6) = \sum_{i=1}^6 x_iy_i$. 
\end{definition}

\begin{lemma}\label{thm:torsion2-5}
  The restriction map $\rho\colon\Tors\Pic(X')\to A_0[2]\times B_0[2]\times
  C_0[2]$ is injective, and the image is identified with the
  orthogonal complement of the subspace generated by the vectors $\vec
  P_s$.
\end{lemma}
\begin{proof}
  The restrictions of the following divisors to $V$ give the subset $B_0[2]$:
  \begin{displaymath}
    0, \quad A_1-A_2 = 00\ 10\ 00, \quad
    A_1-A_3-C_0 = 00\ 11\ 00, \quad A_2 -A_3-C_0 = 00\ 01\ 00.
  \end{displaymath}
  Among these, the divisors containing $A_1$ are precisely those for
  which the vector $v\in B_0[2]\subset V$ satisfies $v\cdot \vec A_1
  =1$. Repeating this verbatim, one has the same results for the
  divisors $A_2,\dotsc, C_2$ and vectors $\vec A_2,\dotsc,\vec C_2$.

  Let $\bar D$ be a linear combination of the divisors $\bar A_1-\bar
  A_2$, $\bar A_1-\bar A_3-\bar C_0$, $\bar A_2-\bar A_3-\bar C_0$,
  and the corresponding divisors for $C_0[2]$, $A_0[2]$. Define the
  vector $v(D)\in V$ to be the sum of the corresponding vectors
  $A_1-A_2 \in V$, etc. 

  Now assume that the vector $v(D)$ satisfies the condition $v(D)\cdot
  \vec P_s=0$ for all the points $P_s$. Then the coefficients of the
  exceptional divisors $\bar E_s$ in the divisor $\epsilon^*(\bar D)$
  on $Y'$ are even (and one can also easily arrange them to be zero
  since the important part is working modulo 2). Therefore, a lift of
  $\eps^*(\bar D)$ to $X'$ is well defined and is a torsion in
  $\Pic(X')$.

  This shows that the image of the homomorphism $\rho\colon\Tors\Pic
  X'\to V$ contains the space $\langle \vec P_s\rangle^\perp$. But
  this space already has the correct dimension. Indeed, for $3\le
  K^2\le 5$ the vectors $\vec P_s$ are linearly independent, and for
  $K^2=2$ the vectors $\vec P_1 = \vec A_1+\vec B_1+\vec C_1$, $\vec P_2 = \vec
  A_1+\vec B_2+\vec C_2$, $\vec P_3=\vec A_2+\vec B_1+\vec C_2$, $\vec P_4 =
  \vec A_2+\vec B_2+\vec C_1$ are linearly dependent (their sum is
  zero) and span a subspace of dimension 3; thus the orthogonal
  complement has dimension 3 as well. Therefore, $\rho$ is a
  bijection of $\Tors \Pic(X')$ onto $\langle \vec P_s\rangle^\perp$.
\end{proof}

\begin{theorem}\label{thm:PicX2-5}
  One has the following:
  \begin{enumerate}
  \item The homomorphism
    \begin{eqnarray*}
      \phi\colon
      \Pic X' &\to& \bZ^{1+k} \times \Pic A'_0 \times \Pic B'_0 \times
      \Pic C'_0 \\ 
      L&\mapsto& (d(L)=L\cdot K_{X'},\  
      L\cdot \frac12 E_s,\ 
      L|_{A'_0},\  L|_{B'_0},\  L|_{C'_0} )
    \end{eqnarray*}
    is injective, and the image is the subgroup of index $3\cdot 2^n$
    in $\bZ^{4+k} \times A'_0[2]\times B'_0[2]\times C'_0[2]$, where
    $n=6-K^2$ for $3\le K^2\le 6$ and $n=3$ for $K^2=2$. 
  \item $\phi$ induces an isomorphism $\Tors(\Pic X')\isoto \langle
    \vec P_s\rangle^\perp \subset A'_0[2]\times B'_0[2]\times C'_0[2]$. 
  \item The curves $A'_i,B'_i,C'_i$, $0\le i\le 3$, generate $\Pic X'$.
  \end{enumerate}
\end{theorem}
\begin{proof}
  (2) is \eqref{thm:torsion2-5} and (1) follows from it. For 
  (3), note that $\Pic X'/\Tors=\Pic Y'$ is generated by the divisors
  $A'_i,B'_i,C'_i$ and that the proof of the previous theorem shows that
  $\Tors\Pic X'$ is generated by certain linear combinations of these
  divisors.
\end{proof}

\section{Effective divisors on Burniat surfaces with $K^2=6$}

Since $\frac12\pi^*$ and $\frac12\pi_*$ provide isomorphisms between
the $\bQ$-vector spaces $(\Pic Y)\otimes\bQ$ and $(\Pic X)\otimes\bQ$,
it is obvious that the cones of effective $\bQ$- or $\bR$-divisors on $X$
and $Y$ are naturally identified. 
In this section, we would like to prove the following description of the
semigroup of effective $\bZ$-divisors:

\begin{theorem}\label{thm:eff-divisors}
  The curves $A_i,B_i,C_i$, $0\le i\le 3$, generate 
  the semigroup of effective $\bZ$-divisors on Burniat surface $X$.
\end{theorem}

We start with several preparatory lemmas.

\begin{lemma}\label{lem:eff-semigroup}
  The semigroup of effective $\bZ$-divisors on $Y$ is generated by the
  $(-1)$-curves $\bar A_0,\bar B_0,\bar C_0,\bar A_3,\bar B_3,\bar
  C_3$.
\end{lemma}
\begin{proof}
  Since $-K_Y$ is ample, the Mori-Kleiman cone $NE_1(Y)$ of effective
  curves in $(\Pic Y)\otimes\bQ$ is generated by extremal rays,
  i.e. the $(-1)$-curves $\bar A_0,\bar B_0,\bar C_0,\bar A_3,\bar
  B_3,\bar C_3$. We claim that moreover the semigroup of integral
  points in $NE_1(Y)$ is generated by these points, i.e. the polytope
  $Q=NE_1(Y)\cap \{C\mid -K_YC=1\}$ is totally generating. The
  vertices 
  of this polytope in $\bR^3$ are $(-1,0,0)$, $(0,-1,0)$, $(0,0,-1)$,
  $(0,1,1)$, $(1,0,1)$, $(1,1,0)$, and the lattice $\Pic Y=\bZ^4$ is
  generated by them. 
  It is a prism over a triangular base, and it is totally
  generating because it can be split into 3 elementary simplices.
\end{proof}

\begin{lemma}\label{lem:nef-semigroup}
  The semigroup of nef $\bZ$-divisors on $Y$ is generated by $f_1$,
  $f_2$, $f_3$, $h_1$, and $h_2$. 
\end{lemma}
\begin{proof}
  Again, for the $\bQ$-divisors this is obvious by MMP: a divisor
  $\bar D$ is nef iff $\bar D\bar F\ge0$ for $\bar F\in \{ \bar
  A_0,\bar B_0,\bar C_0,\bar A_3,\bar B_3,\bar C_3\}$, and the
  extremal nef $\bar D$ divisors correspond to contractions $Y\to Y'$ with
  $\rk\Pic Y'=1$. Another proof: the extremal nef divisors correspond
  to the faces of the triangular prism from the proof of
  Lemma~\ref{lem:eff-semigroup}, and there are 5 of them: 3 sides,
  top, and the bottom.

  Now let $\bar D$ be a nonnegative linear combination $\bar D= \sum
  a_if_i + b_jh_j$ and let us assume that $a_1>0$ (resp. $b_1>0$). Since
  the intersections of $f_1$ (resp. $h_1$) with the curves $F$ above
  are 0 or 1, it follows that $\bar D-f_1$ (resp. $D-h_1$) is also
  nef. We finish by induction. 
\end{proof}

We write the divisors $\bar D$ in $\Pic Y$ using the symmetric
coordinates 
\begin{displaymath}
(d; a_0^0, b_0^0, c_0^0; a_3^0, b_3^0, c_3^0), \text{ where }
d=\bar D(-K_Y),\  a_0^0= \bar D\bar A_0, \dotsc, c_3^0= \bar D\bar
C_3.  
\end{displaymath}

Note that, as in Theorem~\ref{thm:PicX}, $\Pic Y$ and can be described
either as the subgroup of $\bZ^4$ with coordinates $(d; a_0^0, b_0^0,
c_0^0)$ satisfying the congruence $3 | (d+ a_0^0+ b_0^0+ c_0^0)$, or
as the subgroup of $\bZ^4$ with coordinates $(d; a_3^0, b_3^0, c_3^0)$
satisfying the congruence $3 | (d+ a_3^0+ b_3^0+ c_3^0)$.

\begin{lemma}\label{lem:exceptions}
  The function $p_a(\bar D) = \dfrac{\bar D(\bar D+K_Y)}{2}+1$ on the
  set of nef $\bZ$-divisors on $Y$ is strictly
  positive, with the exception of the following divisors, up to
  symmetry:
  \begin{enumerate}
  \item $(2n;n,0,0;n,0,0)$ for $n\ge1$, one has $p_a = -(n-1)$ 
  \item $(2n;n-1,1,0;n-1,1,0)$ for $n\ge1$, one has $p_a=0$.
  \item $(2n+1;n,1,1;n-1,0,0)$ and $(2n+1;n-1,0,0;n,1,1)$ 
    for $n\ge1$, $p_a=0$.
  \item $(6;2,2,2;0,0,0)$ and $(6;0,0,0;2,2,2)$, $p_a=0$. 
  \end{enumerate}
  The divisors in $(1)$ are in the linear system $|nf_i|$, where $f_i$
  is a fiber of one of the three rulings $Y\to \bP^1$. The divisors in
  $(2)$ and $(3)$ are obtained from these by adding a section. The
  divisors in $(4)$ belong to the linear systems $|2h_1|$ and
  $|2h_2|$.
\end{lemma}
\begin{proof}
  Let $\bar D$ be a nef $\bZ$-divisor. By
  Lemma~\ref{lem:nef-semigroup}, we can write $\bar D=\sum n_if_i +
  m_jh_j$ with $n_i,m_j\in\bZ_{\ge0}$. Let us say $n_1>0$. If
  $\bar D=n_1f_1$ then $p_a(\bar D)=-(n_1-1)$. Otherwise, $n_1f_1 +
  g\le \bar D$, where $g=f_j$, $j\ne1$, or $g=h_j$. Then using the
  elementary formula
  \begin{math}
    p_a(\bar D_1 + \bar D_2) = p_a(\bar D_1) + p_a(\bar D_2) + 
    \bar D_1 \bar D_2 -1,
  \end{math}
  we see that $p_a(n_1f_1 + g)=0$. Continuing this by induction and
  adding more $f_j$'s and $h_j$'s, one easily obtains that $p_a(\bar D)>0$
  with the only exceptions listed above. Starting with $m_1 h_1$
  instead of $n_1f_1$ works the same.
\end{proof}

\begin{corollary}\label{cor:exceptions}
  The function $\chi( D) = \dfrac{ D( D-K_X)}{2}+1$ on the
  set of nef $\bZ$-divisors on $Y$ is strictly positive, with the same
  exceptions as above. 
\end{corollary}
\begin{proof}
  Indeed, since $\chi(\cO_X)=1$, one has $\chi(D) = p_a(\bar D)$.   
\end{proof}

\begin{lemma}\label{lem:D+K}
  Assume that $\bar D\ne0$ is a nef divisor on $X$ with $p_a(\bar
  D)>0$. Then the divisor  $\bar D + K_Y$ is effective.
\end{lemma}
\begin{proof}
  One has
  \begin{math}
    \chi(\bar D+K_Y) = \dfrac{(\bar D+K_Y)\bar D}2 + 1 = p_a(\bar D)>0.
  \end{math}
  Since $h^2(\bar D+K_Y)=h^0(-\bar D)=0$, this implies that $h^0(\bar D)>0$.
\end{proof}

\begin{definition}
  We say that an effective divisor $D$ on $X$ is \emph{in minimal
    form} if $DF\ge0$ for the elliptic curves $F\in\{
  A_0,B_0,C_0,A_3,B_3,C_3 \}$, and for the curves among those that satisfy
  $DF=0$, one has $D|_F \ne 0$ in $F[2]$.

If either of these conditions fails then $D-F$ must also be
effective since $F$ is then in the base locus of $|D|$. A minimal form
is obtained by repeating this procedure until it stops or one obtains
a divisor of negative degree, in which case $D$ obviously was not
effective. We do not claim that a minimal form is unique.
\end{definition}

\begin{proof}[Proof of Thm.~\ref{thm:eff-divisors}]
Let $D$ be an effective divisor on $X$. We have to show that it
belongs to the semigroup $\cS= \langle A_i, B_i,C_i,\ 0\le i\le 3
\rangle$. 

\smallskip

\emph{Step 1: One can assume that $D$ is in minimal form}. Obviously.
\smallskip

\emph{Step 2.: The statement is true for $d\le 6$.} There are finitely
many cases here to check. We checked them using a computer script. For
each of the divisors, putting it in minimal form makes it obvious that
it is either in $\cS$ or it is not effective because it has negative
degree, with the exception of the following three divisors, in the
notations of Theorem~\ref{thm:PicX}: $(3; 1\ 10\ 1\ 10\ 1\ 10)$, $(3;
0\ 00\ 0\ 00\ 0\ 00)$, $(3; 1\ 00\ 1\ 00\ 1\ 00)$. The first two
divisors are not effective by \cite[Lemma
5]{alexeev2012derived-categories}. The third one is not effective
because it is $K_X$ and $h^0(K_X)=p_g(X)=0$.

\smallskip

\emph{Step 3: The statement is true for nef divisors of degree $d\ge 7$
which are \emph{not} the exceptions listed in
Lemma~\ref{lem:exceptions}.}

One has $K_X(K_X-D)<0$, so $h^0(K_X-D)=0$ and the condition
$\chi(D)>0$ implies that $D$ is effective. We are going to show that $D$ is
in the semigroup $\cS$.

Consider the divisor $D-K_X$ which modulo torsion is identified with
the divisor $\bar D+K_Y$ on $Y$. By Lemmas~\ref{lem:D+K} and
\ref{lem:eff-semigroup}, $\bar D+K_Y$ is a positive $\bZ$-combination of
$\bar A_0,\bar B_0,\bar C_0,\bar A_3,\bar B_3,\bar C_3$. This means
that
\begin{displaymath}
  D = K_X + \text{(a positive combination of }
  A_0,B_0,C_0,A_3,B_3,C_3 \text{)} 
  + \text{(torsion } \nu \text{)}
\end{displaymath}
A direct computer check shows that for any torsion $\nu$ the divisor
$K_X+F+\nu$ is in $\cS$ for a single curve $F\in \{
A_0,B_0,C_0,A_3,B_3,C_3 \}$. (In fact, for any $\nu\ne0$ the divisor
$K_X+\nu$ is already in $\cS$.)  Thus, 
\begin{displaymath}
  D - \text{(a nonnegative combination of }
  A_0,B_0,C_0,A_3,B_3,C_3  
  \text{)}
  \in \cS 
  \implies D\in \cS.
\end{displaymath}

\smallskip

\emph{Step 4: The statement is true for nef divisors in minimal form
  of degree $d\ge 7$ which \emph{are} the exceptions listed in
  Lemma~\ref{lem:exceptions}.}

We claim that any such divisor is in $\cS$, and in particular is
effective. For $d=7,8$ this is again a direct computer check. For
$d\ge 9$, the claim is true by induction, as follows:
If $D$ is of exceptional type (1,2, or 3) of Lemma~~\ref{lem:exceptions}
then $D-C_1$ has degree $d'=d-2$ and is of the same exceptional
type. This concludes the proof.
\end{proof}

\begin{remark}
  Note that we proved a little more than what
  Theorem~\ref{thm:eff-divisors} says. We also proved that every
  divisor $D$ in minimal form and of degree $\ge7$ is effective and is
  in the semigroup $\cS$.
\end{remark}

\begin{remark}
  For Burniat surfaces with $2\le K^2\le 5$, a natural question to ask
  is whether the semigroup of effective $\bZ$-divisors is generated by
  the preimages of the $(-1)$- and $(-2)$ curves on $Y'$. These
  include the divisors $A'_i,B'_i,C'_i$ and $E_s$ but in some cases
  there are other curves, too.
\end{remark}

\section{Exceptional collections on degenerate Burniat surfaces}
\label{sec:degs}

Degenerations of Burniat surfaces with $K_X^2=6$ were described in
\cite{alexeev2009explicit-compactifications}. Here, we will
concentrate on one particularly nice degeneration depicted in
Figure~\ref{fig:degn}. 

\begin{figure}[h]
  \centering
  \includegraphics[height=1.5in]{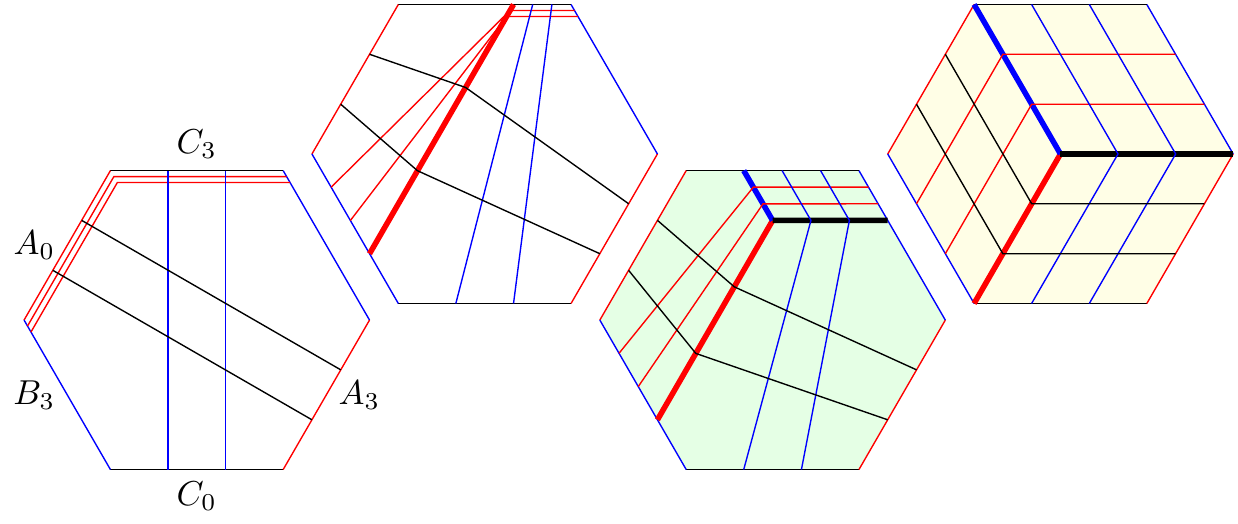}
  \caption{One-parameter degeneration of Burniat surfaces}
  \label{fig:degn} 
\end{figure}

It is described as follows. One begins with a one-parameter family
$f\colon (Y\times\bA^1,\sum_{i=0}^3 \bar A_i+\bar B_i+\bar C_i) \to \bA^1$ of del Pezzo
surfaces, in which the curves degenerate in the
central fiber $f\inv(0)$ to a configuration shown in the left
panel. The surface $\cY$ is obtained from $Y\times\bA^1$ by two blowups in the
central fiber, along the smooth centers $\bar A_0$ and then (the
strict preimage of) $\bar C_3$. The resulting 3-fold $\cY$ is smooth,
the central fiber $\cY_0 = \Bl_3\bP^2\cup \Bl_2\bP^2\cup
(\bP^1\times\bP^1)$ is reduced and has normal crossings. This central
fiber is shown in the third panel.

The log canonical divisor $K_{\cY}+ \frac12\sum_{i=0}^3 (\bar A_i+\bar
B_i+\bar C_i)$ is relatively big and nef over $\bA^1$. It is a
relatively minimal model. The relative canonical model $\cY\ucan$ is
obtained from $\cY$ by contracting three curves. The 3-fold $\cY\ucan$
is singular at three points and not $\bQ$-factorial. Its central fiber
$\cY\ucan_0$ is shown in the last, fourth panel.

The 3-folds $\pi\colon \cX\to \cY$ and $\pi\ucan\colon \cX\ucan\to\cY\ucan$
are the corresponding $\bZ_2^2$-Galois covers. The 3-fold $\cX$ is
smooth, and its central fiber $\cX_0$ is reduced and has normal
crossings. It is a relatively minimal model: $K_{\cX}$ is relatively
big and nef. 

The 3-fold $\cX\ucan$ is obtained from $\cX$ by contracting three
curves. Its canonical divisor $K_{\cX\ucan}$ is relatively ample. It
is a relative canonical model. We note that $\cX$ is one of the 6
relative minimal models $\cX^{(k)}$, $k=1,\dotsc,6$, that are related
by flops.  

Let $U\subset \bA^1$ be the open subset containing $0$ and all $t$ for
which the fiber $\cX_t$ is smooth, and let $\cX_U = \cX\times_{\bA^1}U$.
The aim of this section is to prove the following:

\begin{theorem}\label{thm:exc-colln-singular}
  Then there exists a sequence of line bundles $\cL_1,\dotsc,\cL_6$ on
  $\cX_U$ whose restrictions to any fiber (including the nonnormal
  semistable fiber $\cX_0$) form and exceptional collection of line
  bundles. 
\end{theorem}

\begin{remark}
  It seems to be considerably harder to construct an exceptional
  collection on the surface $\cX\ucan_0$, the special fiber in a
  singular 3-fold $\cX\ucan$. And perhaps looking for one is not the
  right thing to do. A familiar result is that different smooth
  minimal models $\cX^{(k)}$ related by flops have equivalent derived
  categories. In the same vein, in our situation the central fibers
  $\cX^{(k)}_0$, which are reduced reducible semistable varieties,
  should have the same derived categories. The collection we construct
  works the same way for any of them.
\end{remark}

\begin{notation}
  On the surface $\cX_0$, we have 12 Cartier divisors
  $A_i,B_i,C_i$, $i=0,1,2,3$. The ``internal'' divisors $A_i,B_i,C_i$,
  $i=1,2$ have two irreducible components each. Of the 6 ``boundary''
  divisors, $A_0,A_3,C_0$ are irreducible, and $B_0=B_0'+B_0''$,
  $B_3=B_3'+B_3''$, $C_3=C_3'+C_3''$ are reducible. 

  Our notation for the latter divisors is as follows: the curve $C_3'$
  is a smooth elliptic curve (on the bottom surface $(\cY)_0$ the
  corresponding curve has 4 ramification points), and the curve $C_3''$ is
  isomorphic to $\bP^1$ (on the bottom surface the corresponding curve
  has 2 ramification points). 

  For consistency of notation, we also set $A_0'=A_0$, $A_3'=A_3$,
  $C_0'=C_0$. 
\end{notation}

\begin{definition}
  Let $\psi=\psi_{C_3}\colon C_3\to C_3'$ be the projection which is an
  isomorphism on the component $C_3'$ and collapses the component
  $C_3''$ to a point. 

  We have natural norm map $\psi_*=(\psi_{C_3})_*\colon \Pic C_3\to \Pic
  C_3'$. Indeed, every line bundle on the reducible curve $C_3$ can be
  represented as a Cartier divisor $\cO_{C_3}(\sum n_iP_i)$, where
  $P_i$ are nonsingular points. Then we define 
  \begin{displaymath}
    \psi_* \big( \cO_{C_3}(\sum n_iP_i) \big ) =
    \cO_{C_3'}( \sum n_i \psi(P_i) ).
  \end{displaymath}
  Since the dual graph of the curve $C_3$ is a tree, one has
  $\Pic^0 C_3 =\Pic^0 C_3'$ and $\Pic C_3 = \Pic^0C_3' \oplus \bZ^2$. 

  We also have similar morphisms $\psi_{B_0}$,
  $\psi_{B_3}$ and norm maps for the other two reducible curves. 
\end{definition}

\begin{definition}
  We define a map $\phi_{C_3}\colon \Pic \cX_0 \to \Pic C_3'$
  as the composition of the restriction to $C_3$ and the norm map
  $\psi_*\colon C_3\to C_3'$.   We also have similar morphisms $\phi_{B_0}$,
  $\phi_{B_3}$ for the other two reducible curves. For the irreducible
  curves $A_0,A_3,C_0$ the corresponding maps are simply the restriction
  maps on Picard groups. 
\end{definition}

For the following Lemma, compare Theorem~\ref{thm:PicX} above. 

\begin{lemma}\label{lem:restrictions}
  Consider the map
  \begin{displaymath}
    \phi_0\colon 
    \Pic \cX_0\to \bZ\oplus \Pic A_0' \oplus \Pic B_0' \oplus \Pic C_0'
  \end{displaymath}
  defined as $D\mapsto \deg (DK_{\cX_0})$ in the first component and 
  the maps $\phi_{A_0}$, $\phi_{B_0}$, $\phi_{C_0}$ in the other
  components. Then the images of the Cartier divisors $A_i,B_i,C_i$,
  $i=0,1,2,3$ are exactly the same as for a smooth Burniat surface
  $\cX_t$, $t\ne0$. 
\end{lemma}
\begin{proof}
  Immediate check.
\end{proof}

\begin{definition}
  We will denote this image by $\im\phi_0$. One has $\im\phi_0\simeq
  \bZ^4\oplus \bZ_2^6$. We emphasize that $\im\phi_0=\im\phi_t=\Pic
  \cX_t$, where $\cX_t$ is a smooth Burniat surface. 

\end{definition}

\begin{lemma}\label{lem:subtract1}
  Let $D$ be an effective Cartier divisor $D$ on the surface
  $\cX_0$. Suppose that $D.A_i<0$ for $i=0,3$. Then the Cartier
  divisor $D-A_i$ is also effective. (Similarly for $B_i,C_i$.)
\end{lemma}
\begin{proof}
  For an irreducible divisor this is immediate, so let us do it for
  the divisor $C_3=C_3'+C_3''$ which spans two irreducible components,
  say $X',X''$ of the surface $\cX_0=X'\cup X''\cup X'''$. 
  Let $D'=D|_{X'}$, $D''=D|_{X''}$, $D'''=D|_{X'''}$. Then
  \begin{displaymath}
    D.C_3= (D'.C_3')_{X'} + (D''.C_3'')_{X''},    
  \end{displaymath}
  where the right-hand intersections are computed on the smooth
  irreducible surfaces. One has $(C_3')^2_{X'}=0$ and
  $(C_3'')^2_{X''}=-1$. Therefore, $(D'.C_3')_{X'}\ge0$. Thus,
  $D.C_3<0$ implies that $(D''.C_3'')_{X''}<0$. Then $C_3''$ must be
  in the base locus of the linear system $|D''|$ on the smooth surface
  $X''$. Let $n>0$ be the multiplicity of $C_3''$ in $D''$. Then the
  divisor $D''-nC_3''$ is effective and does not contain $C_3''$. 

  By what we just proved, $D$ must contain
  $nC_3''$. Thus, it passes through the point $P=C_3'\cap C_3''$ and
  the multiplicity of the curve $(D')_{X'}$ at $P$ is $\ge n$, since $D$
  is a Cartier divisor. 
  Suppose that $D$ does not contain the curve $C_3'$. Then 
  $(D'.C_3')_{X'}\ge n$, and 
  \begin{displaymath}
    D.C_3 = (D'.C_3')_{X'}+(D''.C_3'')_{X''} \ge n + (-n) =0,
  \end{displaymath}
  which provides a contradiction. We conclude that $D$ contains $C_3'$
  as well, and so $D-C_3$ is effective. 
\end{proof}

\begin{lemma}\label{lem:subtract2}
  Let $D$ be an effective Cartier divisor $D$ on the surface
  $\cX_0$. Suppose that $D.A_i=0$ for $i=0,3$ but $\phi_{A_i}(D)\ne
  0$ in $\Pic A_i$.  Then the Cartier divisor $D-A_i$ is also
  effective.  (Similarly for $B_i,C_i$.)
\end{lemma}
\begin{proof}
  We use the same notations as in the proof of the previous lemma. 
  Since $D'$ is effective, one has $(D'.C_3')_{X'}\ge0$. 

  If $(D''.C_3'')_{X''}<0$ then, as in the above proof one must have
  $D''=nC_3''$ and $D'$ intersect $C_3'$ only at the unique point
  $P=C_3'\cap C_3''$ and $(D'.C_3')_{X'}=n$. But then
  $\phi_{C_3}(D)=0$ in $\Pic C_3'$, a contradiction.

  If $(D''.C_3'')_{X''}=0$ but $D''-nC_3''$ is effective for some
  $n>0$, the same argument gives $DC_3>0$, so an even easier contradiction.

  Finally, assume that $(D'.C_3')_{X'}=(D''.C_3'')_{X''}=0$ and $D''$
  does not contain $C_3''$.  By assumption, we have $D'C_3'=0$ but
  $D'|_{C_3'}\ne0$ in $\Pic C_3'$. This implies that $D'-C_3'$ is
  effective and that $D$ contains the point $P=C_3'\cap C_3''$. But
  then $(D''.C_3'')_{X''}>0$. Contradiction.
\end{proof}

The following lemma is the precise analogue of \cite[Lemma
5]{alexeev2012derived-categories} (Lemma~4.5 in 
  the arXiv version).

\begin{lemma}\label{lem:AO-4.5}
  Let $F\in \Pic \cX_0$ be an invertible sheaf such that
  \begin{displaymath}
    \im \phi_0(F) = (3; 1\ 10, 1\ 10, 1\ 10) 
    \in \bZ\oplus \Pic A_0 \oplus \Pic B_0 \oplus C_0
  \end{displaymath}
  Then $h^0(\cX_0, F)=0$. 
\end{lemma}
\begin{proof}
  The proof of \cite[Lemma 5]{alexeev2012derived-categories}, used
  verbatim together with the above Lemmas~\ref{lem:subtract1},
  \ref{lem:subtract2} works. Crucially, the three ``corners'' $A_0\cap
  C_3$, $B_0\cap A_3$, $C_0\cap B_3$ are smooth points on $\cX_0$. 
\end{proof}

\begin{proof}[Proof of Thm.~\ref{thm:exc-colln-singular}] 
  We define the sheaves $\cL_1,\dotsc,\cL_6$ by the same linear
  combinations of the Cartier divisors $\cA_i,\cB_i,\cC_i$ as in the
  smooth case \cite[Rem.2]{alexeev2012derived-categories} (Remark 4.4
  in the arXiv version), namely:
  \begin{eqnarray*}
&    \cL_1=\cO_\cX(\cA_3+\cB_0+\cC_0+\cA_1-\cA_2),
&    \cL_2=\cO_\cX(\cA_0+\cB_3+\cC_3+\cC_2-\cA_1), \\
&    \cL_3=\cO_\cX(\cC_2+\cA_2-\cC_0-\cA_3), 
&    \cL_4=\cO_\cX(\cB_2+\cC_2-\cB_0-\cC_3), \\
&    \cL_5=\cO_\cX(\cA_2+\cB_2-\cA_0-\cB_3), 
&    \cL_6=\cO_\cX.
  \end{eqnarray*}

  By \cite{alexeev2012derived-categories}, for every $t\ne0$
  they restrict to the invertible sheaves
  $L_1,\dotsc, L_6 \in \im\phi_t = \Pic \cX_t$ on a smooth Burniat
  surface which form an
  exceptional sequence.
  By Lemma~\ref{lem:restrictions}, the images of $\cL_i|_{\cX_0} \in
  \Pic \cX_0$ under the map
  \begin{displaymath}
    \phi_0\colon \Pic\cX_0 \onto \im\phi_0= \im\phi_t = \Pic\cX_t, 
    \quad t\ne 0.
  \end{displaymath}
  are also $L_1,\dotsc,L_6$. We claim that $\cL_i|_{\cX_0}$ also form an
  exceptional collection.

  Indeed, the proof in \cite{alexeev2012derived-categories} of the
  fact that $L_1,\dotsc, L_6$ is an exceptional collection on a smooth
  Burniat surface $\cX_t$ ($t\ne0$) consists of showing that for
  $i<j$ one has
  \begin{enumerate}
  \item $\chi(L_i\otimes L_j\inv)=0$,
  \item $h^0(L_i\otimes L_j)=0$, and 
  \item $h^0(K_{\cX_t}\otimes L_i\inv\otimes L_j)=0$. 
  \end{enumerate}
  The properties (2) and (3) are checked by repeatedly applying (the
  analogues of) Lemmas~\ref{lem:subtract1}, \ref{lem:subtract2},
  \ref{lem:AO-4.5} until $D.K_{\cX_t}<0$ (in which case $D$ is
  obviously not effective).

  In our case, one has $\chi(\cX_0, \cL_i|_{\cX_0}\otimes
  \cL_j|_{\cX_0}\inv) = \chi(\cX_t,\cL_i|_{\cX_t}\otimes
  \cL_j|_{\cX_t})\inv=0$ by flatness.  Since we proved that
  Lemmas~\ref{lem:subtract1}, \ref{lem:subtract2}, \ref{lem:AO-4.5}
  hold for the surface $\cX_0$, and since the Cartier divisor
  $K_{\cX_0}$ is nef, the same exact proof for vanishing of $h^0$ goes
  through unchanged.
\end{proof}

\begin{remark}
  The semiorthogonal complement $\cA_t$ of the full triangulated
  category generated by the sheaves $\langle
  \cL_1,\dotsc,\cL_6\rangle|_{\cX_t}$ is the quite mysterious
  ``quasiphantom''. A viable way to understand it could be to
  understand the degenerate quasiphantom $\cA_0=\langle
  \cL_1,\dotsc,\cL_6\rangle|_{\cX_t}^\perp$ on the semistable
  degeneration $\cX_0$ first. The irreducible components of $\cX_0$
  are three bielliptic surfaces and they are glued nicely. Then one
  could try to understand $\cA_t$ as a deformation of $\cA_0$. 
\end{remark}

\bibliographystyle{amsalpha} \renewcommand{\MR}[1]{}

\def\cprime{$'$}
\providecommand{\bysame}{\leavevmode\hbox to3em{\hrulefill}\thinspace}
\providecommand{\MR}{\relax\ifhmode\unskip\space\fi MR }
\providecommand{\MRhref}[2]{%
  \href{http://www.ams.org/mathscinet-getitem?mr=#1}{#2}
}
\providecommand{\href}[2]{#2}

\end{document}